\UseRawInputEncoding
\documentclass[12pt]{article}

\usepackage{dsfont}
\usepackage{amsfonts,amsmath,amsthm}
\usepackage{algorithm, algorithmic}
\usepackage{color}
\usepackage{graphicx}
\usepackage{stmaryrd}        

\usepackage[paper=a4paper,dvips,top=2cm,left=2cm,right=2cm,
foot=1cm,bottom=4cm]{geometry}

\newcommand{\ve}{{\bf e}}

\begin{document}
	\large
	
	\title{Eigenvalues {and Jordan Forms} of Dual Complex Matrices}
	\author{ Liqun Qi\footnote{Department of Mathematics, School of Science, Hangzhou Dianzi University, Hangzhou 310018 China; Department of Applied Mathematics, The Hong Kong Polytechnic University, Hung Hom, Kowloon, Hong Kong
			({\tt maqilq@polyu.edu.hk}).}
		\and \
		Chunfeng Cui\footnote{LMIB of the Ministry of Education, School of Mathematical Sciences, Beihang University, Beijing 100191 China.
			({\tt chungfengcui@buaa.edu.cn}).}
	}
	\date{\today}
	\maketitle

	\begin{abstract}
		Dual complex matrices have found applications in brain science.   There are two different definitions of the dual complex number multiplication.  One is noncommutative.  Another is commutative.   In this paper, we use the commutative definition.  This definition is used in the research related with brain science.
		Under this definition, eigenvalues of dual complex matrices are defined.   However, there are cases of dual complex matrices which have no eigenvalues or have infinitely many eigenvalues.   We show that an $n \times n$ dual complex matrix is diagonalizable if and only if it has exactly $n$ eigenvalues with $n$ appreciably linearly independent eigenvectors.   Hermitian dual complex matrices are diagonalizable.   We present the Jordan form of a dual complex matrix with a diagonalizable standard part, and the Jordan form of a dual complex matrix with a Jordan block standard part.   Based on these, we 
		give a description of the eigenvalues of a general square dual complex matrix.

		\medskip


		
	\end{abstract}

	\renewcommand{\Re}{\mathds{R}}
	\newcommand{\rank}{\mathrm{rank}}
	\renewcommand{\span}{\mathrm{span}}
	\newcommand{\X}{\mathcal{X}}
	\newcommand{\A}{\mathcal{A}}
	\newcommand{\I}{\mathcal{I}}
	\newcommand{\B}{\mathcal{B}}
	\newcommand{\C}{\mathcal{C}}
	\newcommand{\OO}{\mathcal{O}}
	\newcommand{\e}{\mathbf{e}}
	\newcommand{\0}{\mathbf{0}}
	\newcommand{\dd}{\mathbf{d}}
	\newcommand{\ii}{\mathbf{i}}
	\newcommand{\jj}{\mathbf{j}}
	\newcommand{\kk}{\mathbf{k}}
	\newcommand{\va}{\mathbf{a}}
	\newcommand{\vb}{\mathbf{b}}
	\newcommand{\vc}{\mathbf{c}}
	\newcommand{\vq}{\mathbf{q}}
	\newcommand{\vg}{\mathbf{g}}
	\newcommand{\pr}{\vec{r}}
	\newcommand{\ps}{\vec{s}}
	\newcommand{\pt}{\vec{t}}
	\newcommand{\pu}{\vec{u}}
	\newcommand{\pv}{\vec{v}}
	\newcommand{\pw}{\vec{w}}
	\newcommand{\pp}{\vec{p}}
	\newcommand{\pq}{\vec{q}}
	\newcommand{\pl}{\vec{l}}
	\newcommand{\vt}{\rm{vec}}
	\newcommand{\vx}{\mathbf{x}}
	\newcommand{\vy}{\mathbf{y}}
	\newcommand{\vu}{\mathbf{u}}
	\newcommand{\vv}{\mathbf{v}}
	\newcommand{\y}{\mathbf{y}}
	\newcommand{\vz}{\mathbf{z}}
	\newcommand{\T}{\top}
	
	\newtheorem{Thm}{Theorem}[section]
	\newtheorem{Def}[Thm]{Definition}
	\newtheorem{Ass}[Thm]{Assumption}
	\newtheorem{Lem}[Thm]{Lemma}
	\newtheorem{Prop}[Thm]{Proposition}
	\newtheorem{Cor}[Thm]{Corollary}
	\newtheorem{example}[Thm]{Example}
	\newtheorem{remark}[Thm]{Remark}
	
	\section{Introduction}
	
	Dual complex matrices and their singular value decomposition theory have found application in brain science \cite{QACLL22, WDW23}.    There are two approaches to consider the eigenvalue theory of general dual complex matrices.  One approach is contained in \cite{QL21}.  In this approach, dual complex multiplication is defined as noncommutative.   The motivation for this is to represent rigid body motion in the plane \cite{MKO14}.   Following this approach, the eigenvalue theory of dual complex matrices is somewhat complicated.  As the multiplication is noncommutative, one only can define right and left eigenvalues as the quaternion matrix case \cite{Ji19, WLZZ18, Zh97}.   
	We do not use this approach here.   Another approach is to consider dual complex numbers as a special of dual quaternion numbers.   This approach is contained in \cite{QACLL22} and \cite{WDW23}. This approach has found applications in brain science.   We follow this approach here.   Eigenvalues of dual complex matrices were introduced in  \cite{QACLL22}.  Then it turned to discuss eigenvalues of dual complex Hermitian matrices, but there was no further discussion on eigenvalues of general dual complex square matrices in \cite{QACLL22}.   We now follow this approach here.  
	
	It turns out that the eigenvalue theory of dual complex matrices is nontrivial.   In Section 3, we will give examples that a dual complex matrix has no eigenvalues at all, and another dual complex matrix has infinitely many eigenvalues.  This {shows} that the eigenvalue theory of dual complex matrices is not {a simple duplicate of} the eigenvalue theory of complex matrices \cite{HJ12}, where such things never happen.
	
	Fortunately, such situations do not always happen for dual complex matrices. Dual complex matrices are the special cases of dual quaternion matrices.   By Qi and Luo \cite{QL23}, an $n \times n$ dual complex Hermitian matrix has exactly $n$ dual number eigenvalues.   This matrix is positive semidefinite (definite) if and only if these $n$ dual number eigenvalues are nonnegative (positive), in the sense of \cite{QLY22}.
	
	In fact, we can show that a diagonalizable $n \times n$ dual complex matrix has exactly $n$ eigenvalues.
	
	All of these indicate that we need to fully characterize the situation of eigenvalues of dual complex matrices: under which conditions an $n \times n$  dual complex matrix has exactly $n$ eigenvalues, no eigenvalue at all, and infinitely many eigenvalues.
	
	This study also paves a way for further studying the eigenvalue theory of dual quaternion matrices \cite{LHQ22, LHQF23, LQY23, QL23}, which has applications in the multi-agent formation control \cite{QWL23} and simultaneous location and mapping \cite{CQ23}.
	
	In Section 2, we review some basic knowledge on dual complex numbers and dual complex vectors.
	
	In Section 3, we present examples that that a dual complex matrix has no eigenvalues at all, and another dual complex matrix has infinitely many eigenvalues.   A dual complex matrix has an eigenvalue with an eigenvector only if the standard part of that eigenvalue, and the standard part of that eigenvector are the eigenvalue and respective eigenvector of the standard part of that dual complex matrix.  Then in that situation we give conditions under which such an eigenvalue and an eigenvector of that dual complex matrix exist, and conditions under which such an eigenvalue and an eigenvector are unique.
	
	In Section 4, we show that an $n \times n$ dual complex matrix is diagonalizable if and only if it has exactly $n$ eigenvalues with $n$ appreciably linearly independent eigenvectors.
	
	{In Section 5, we} present the Jordan form of a dual complex matrix with a diagonalizable standard part, {give another necessary and sufficient condition for an $n \times n$ dual complex matrix to be diagonalizable, and use two more examples to show that under the necessary condition that the standard part of an $n \times n$ dual complex matrix is diagonalizable, the condition that its dual part is diagonalizable is neither necessary nor sufficient for that dual complex matrix to be diagonalizable.}
	
	{Then we present} the Jordan form of a dual complex matrix with a Jordan block standard part in Section 6.
	
	In Section 7, we 
	give a description of the eigenvalues of a general square dual complex matrix.
	
	Some further discussion is conducted in  Section 8.

	\section{Dual Complex Numbers}
	
	The field of real numbers, the field of complex numbers, the set of dual numbers, and the set of dual complex numbers are denoted  by $\mathbb R$, $\mathbb C$, $\mathbb D$ and $\mathbb {DC}$, respectively.
	
	A {\bf dual complex number} $a = a_s + a_d\epsilon$ has standard part $a_s$ and dual part $a_d$.   Both $a_s$ and $a_d$ are complex numbers.   The symbol $\epsilon$ is the infinitesimal unit, satisfying $\epsilon^2 = 0$, and $\epsilon$ is commutative with complex numbers.  If $a_s \not = 0$, then we say that $a$ is {\bf appreciable}.    If $a_s$ and $a_d$ are real numbers, then $a$ is called a {\bf dual number}.
	
	The conjugate of $a = a_s + a_d\epsilon$ is $a^* = a_s^* + a_d^*\epsilon$.
	
	Suppose we have two dual complex numbers $a = a_s + a_d\epsilon$ and $b = b_s + b_d\epsilon$.   Then their sum is $a+b = (a_s+b_s) + (a_d+b_d)\epsilon$, and their product is $ab = a_sb_s + (a_sb_d+a_db_s)\epsilon$.
	The multiplication of dual complex numbers is commutative.
	
	Suppose we have two dual numbers $a = a_s + a_d\epsilon$ and $b = b_s + b_d\epsilon$.   By \cite{QLY22}, if $a_s > b_s$, or $a_s = b_s$ and $a_d > b_d$, then we say $a > b$.   Then this defines positive, nonnegative dual numbers, etc. Some applications of this total order of dual numbers can be founded in \cite{CQ23, LHQ22, LHQF23, LQY23, QL23, QWL23, WCW23}.
	
	In particular, for a dual complex number $a = a_s + a_d\epsilon$, its magnitude is defined as a nonnegative dual number $|a| = |a_s| + |a_d|\epsilon$.  {If a dual number $a = a_s + a_d\epsilon$ is nonnegative and appreciable, then the square root of $a$ is still a nonnegative dual number.   If $a$ is positive and appreciable, we have
		$$\sqrt{a} = \sqrt{a_s} + {a_d \over 2\sqrt{a_s}}\epsilon.$$
		When $a=0$, we have $\sqrt{a} = 0$.}
	
	A dual complex number vector is denoted by $\vx = (x_1, \cdots, x_n)^\top \in {\mathbb {DC}}^n$.  Its $2$-norm is defined as
	$$\|\vx\|_2 = \left\{\begin{array}{ll}
		\sqrt{\sum_{i=1}^n |x_i|^2}, & \ \mathrm{if} \  \vx_s \not = \0,\\
		\|\vx_d\|_2\epsilon, & \ \mathrm{if} \  \vx_s  = \0.
	\end{array}\right.$$
	We may denote $\vx = \vx_s + \vx_d\epsilon$, where $\vx_s, \vx_d \in {\mathbb C}^n$.   If $\vx_s \not = \0$, then we say that $\vx$ is appreciable.   The unit vectors in ${\mathbb R}^n$ are denoted as $\ve_1, \cdots, \ve_n$.   They are also unit vectors of ${\mathbb {DC}}^n$.
	
	For $\vx = (x_1, \cdots, x_n)^\top \in {\mathbb {DC}}^n$, denote $\vx^* = (x_1^*, \cdots, x_n^*)$.  Let $\vy = (y_1, \cdots, y_n)^\top \in {\mathbb {DC}}^n$.   Define
	$$\vx^*\vy = \sum_{j=1}^n x_j^*y_j.$$
	If $\vx^*\vy = 0$, we say that $\vx$ and $\vy$ are {\bf orthogonal}.  Note that $\vx^*\vx = \|\vx\|_2^2$.
	Let $\vx_1, \cdots, \vx_n \in {\mathbb {DC}}^n$.   If $\vx_i^*\vx_j = 0$ for $i \not = j$ and $\vx_i^*\vx_j = 1$ for $i = j$, $i, j = 1, \cdots, n$, then we say that $\vx_1, \cdots, \vx_n$ form an orthonormal basis of ${\mathbb {DC}}^n$.   

	\bigskip

	\section{Eigenvalues of Dual Complex Matrices}
	
	In this section, we always assume
	that $A = A_s + A_d\epsilon$ and $B = B_s + B_d\epsilon$ are two $n \times n$ dual complex matrices, where $n$ is a positive integer, $A_s, A_d, B_s$ and $B_d$ are four complex matrices, assume that $\lambda = \lambda_s + \lambda_d\epsilon$ and $\mu = \mu_s + \mu_d \epsilon$ are two dual complex numbers with $\lambda_s, \lambda_d, \mu_s$ and $\mu_d$ being complex numbers, and assume that
	$\vx = \vx_s + \vx_d \epsilon$ and $\vy = \vy_s + \vy_d \epsilon$ are two $n$-dimensional dual complex vectors with $\vx_s, \vx_d, \vy_s$ and $\vy_d$ being $n$-dimensional {complex vectors}.
	
	If $AB = BA = I$, where
	$I$ is the $n \times n$ identity matrices, then we say that $B$ is the {\bf inverse} of $A$ and denote that $B = A^{-1}$.  The following proposition can be proved directly by definition.
	
	\begin{Prop} \label{p3.1}
		Suppose that $A = A_s + A_d\epsilon$ and $B = B_s + B_d\epsilon$ are two $n \times n$ dual complex matrices, where $n$ is a positive integer, $A_s, A_d, B_s$ and $B_d$ are four complex matrices.   Then the following four statements are equivalent.
		
		(a) $B = A^{-1}$;
		
		(b) $AB = I$;

		(c) $A_sB_s = I$ and $A_sB_d+A_dB_s=O$;
		
		
		(d) $B_s = A_s^{-1}$ and $B_d= -A_s^{-1}A_dA_s^{-1}$.
	\end{Prop}
	
	If there is an $n \times n$ invertible dual complex matrix $P$ such that $A = P^{-1}BP$, then we say that $A$ and $B$ are {\bf similar}, and denote $A \sim B$.
	
	For an $n \times n$ dual complex matrix $A$, denote its conjugate transpose as $A^*$.   If $A^* = A^{-1}$, then $A$ is called a {\bf unitary} matrix.  If a dual number matrix is a unitary matrix, then we simply call it an {\bf orthogonal matrix}.
	
	If
	\begin{equation} \label{en1}
		A\vx = \lambda\vx,
	\end{equation}
	where $\vx$ is appreciable, i.e., $\vx_s \not = \0$, then $\lambda$ is called an eigenvalue of $A$, with an eigenvector $\vx$.
	
	By definition, if $A \sim B$, then $A$ and $B$ have the same eigenvalue set.  In fact, we have the following proposition.
	
	\begin{Prop} \label{p3.2}
		Suppose that $A, B \in {\mathbb {DC}}^{n \times n}$, $A \sim B$, i.e., $A = P^{-1}BP$ for some invertible matrix $P \in {\mathbb {DC}}^{n \times n}$, and $\lambda \in {\mathbb {DC}}$ is an eigenvalue of $A$ with an eigenvector $\vx \in {\mathbb {DC}}^n$.   Then $\lambda$ is an eigenvalue of $B$ with an eigenvector $P\vx$.
	\end{Prop}
	\begin{proof}
		By Proposition \ref{p3.1}, we may show that $P\vx$ is appreciable.  The conclusion now follows directly.
	\end{proof}
	
	Since $A = A_s + A_d\epsilon$, $\lambda = \lambda_s + \lambda_d \epsilon$ and $\vx = \vx_s + \vx_d \epsilon$,
	(\ref{en1}) is equivalent to
	\begin{equation} \label{en2}
		A_s\vx_s = \lambda_s\vx_s,
	\end{equation}
	with $\vx_s \not = \0$, i.e., $\lambda_s$ is an eigenvalue of $A_s$ with an eigenvector $\vx_s$, and
	\begin{equation} \label{en3}
		(A_s-\lambda_sI)\vx_d - \lambda_d\vx_s = -A_d\vx_s.
	\end{equation}

	\medskip
	
	By \cite{QL23}, we know that if $A$ is Hermitian, then it has exactly $n$ dual number eigenvalues, with orthonormal
	eigenvectors, and $A$ is positive semidefinite (definite) if and only if its eigenvalues are nonnegative (positive).
	
	\medskip
	
	{\bf Example 1} - a dual number matrix $A$, which has no eigenvalue at all.
	
	Let $A = A_s+A_d\epsilon$, where
	\begin{equation}
		A_s = \left(\begin{array}{cc}
			1 & 1 \\
			0 & 1
		\end{array} \right)
		\text{ and }
		A_d = \left(\begin{array}{cc}
			0 & 0 \\
			1 & 0
		\end{array} \right).
	\end{equation}
	Then all possible eigenpairs of $A_s$ are: $\lambda_s=1$, $\vx_s=[\alpha,0]^\top$, where $\alpha \not = 0$. Then (\ref{en3}) is equivalent to
	\begin{equation}
		\left(\begin{array}{c}
			x_{d,2}\\
			0
		\end{array} \right) =\alpha \left(\begin{array}{c}
			\lambda_d\\
			-1
		\end{array} \right),
	\end{equation}
	where $\alpha \not = 0$.
	Hence, (\ref{en3}) has no solution for all possible eigenpairs of $A_s$, i.e., $A$ has no eigenvalue at all.
	
	\medskip
	
	{\bf Example 2}  - a dual number matrix $A$, which has infinitely many eigenvalues.

	Let $A = A_s+A_d\epsilon$, where
	\begin{equation}
		A_s = \left(\begin{array}{cc}
			1 & 1 \\
			0 & 1
		\end{array} \right)
		\text{ and }
		A_d = \left(\begin{array}{cc}
			1 & 0 \\
			0 & 0
		\end{array} \right).
	\end{equation}
	$A_s$ has an eigenpair: $\lambda_s=1$, $\vx_s=[1,0]^\top$. Then (\ref{en3}) is equivalent to
	\begin{equation}
		\left(\begin{array}{c}
			x_{d,2}\\
			0
		\end{array} \right) = \left(\begin{array}{c}
			\lambda_d-1\\
			0
		\end{array} \right).
	\end{equation}
	Hence, $\lambda_d= x_{d,2}+1$. As $x_{d,2}$ can be any complex number even if we require $\vx$ to have unit 2-norm, $\lambda_d$ can be any complex number.  Thus, $A$ has infinitely many eigenvalues.
	
	\begin{Prop} \label{p3.3}
		Suppose that $A$ is an $n \times n$ dual complex diagonalizable matrix, i.e., $A$ is similar to a dual complex diagonal matrix.   Then $A$ has exactly $n$ eigenvalues.
	\end{Prop}
	\begin{proof}
		Suppose that $A = PDP^{-1}$, where $D = {\rm diag}(\lambda_1, \cdots, \lambda_n)$, $\lambda_1, \cdots, \lambda_n$ are dual complex numbers.   By definition, $D$ has exactly $n$ eigenvalues $\lambda_1, \cdots, \lambda_n$, with eigenvectors $\ve_1, \cdots, \ve_n$.   By Proposition \ref{p3.2}, $A$ has exactly $n$ eigenvalues $\lambda_1, \cdots, \lambda_n$, with eigenvectors $P\ve_1, \cdots, P\ve_n$.
	\end{proof}
	
	Following the arguments of \cite{QL23}, we may show that an $n \times n$ dual complex Hermitian matrix $A$ is diagonalizable in the sense that $A = PDP^{-1}$, where $D$ is an $n \times n$ dual number diagonal matrix, while $P$ is a dual complex unitary matrix.
	
	Examples 1 and 2 display two situations very different from the situation displayed by Proposition \ref{p3.3}. Then, two questions naturally arise:  When do the situations displayed in Examples 1 and 2 happen?
	Are there any other situations?   We have the following theorem.
	
	\begin{Thm} \label{t3.4}
		Suppose that $A = A_s + A_d\epsilon \in {\mathbb {DC}}^{n \times n}$, $\lambda_s \in {\mathbb C}$ be an eigenvalue of $A_s$.  Then there exist $\lambda_d \in {\mathbb C}$ and $\vx_s, \vx_d \in {\mathbb C}^n$ such that $\lambda = \lambda_s + \lambda_d\epsilon$ is an eigenvalue of $A$ with an eigenvector $\vx = \vx_s + \vx_d\epsilon$, if and only if $\vx_s$ is an eigenvector of $A_s$, and
		\begin{equation} \label{en8}
			A_d\vx_s \in {\rm Span}(Q),
		\end{equation}
		where $Q = (A_s-\lambda_sI_n\ \ \vx_s)$, $I_n$ is the $n \times n$ identity matrix.   If furthermore we have
		\begin{equation} \label{en9}
			\vx_s \not \in {\rm Span}(A_s-\lambda_sI_n),
		\end{equation}
		then $\lambda_d$, hence $\lambda$,  is unique with such $\vx_s$.   Otherwise, $\lambda_d$ can be any complex number.
	\end{Thm}
	\begin{proof}
		We may write (\ref{en3}) as
		\begin{equation} \label{en10}
			{Q\vz} = -A_d\vx_s,
		\end{equation}
		where
		$$\vz = \left({\vx_d \atop -\lambda_d}\right).$$
		Then the conclusions of this proposition follow from the solution existence and uniqueness results of the linear equation system (\ref{en10}).
	\end{proof}
	

	
	\bigskip
	
	\section{Diagonalization of Dual Complex Matrices}
	
	In this section, we establish a necessary and sufficient condition such that an $n \times n$ dual complex matrix is diagonalizable.
	
	Suppose that $\vx_1 = \vx_{1s}+\vx_{1d}\epsilon, \cdots, \vx_k = \vx_{ks}+\vx_{kd}\epsilon \in {\mathbb {DC}}^n$.  We say that $ \vx_1, \cdots, \vx_k $ are appreciably linearly independent if $ \vx_{1s}, \cdots, \vx_{ks} $ are linearly independent.
	
	\begin{Thm} \label{t4.1}
		Suppose that $A \in {\mathbb {DC}}^{n \times n}$.     Then $A$ is diagonalizable if and only if $A$ has exactly $n$ eigenvalues $\lambda_1, \cdots, \lambda_n$ with eigenvectors $\vx_1, \cdots, \vx_n$, which are appreciably linearly independent.
	\end{Thm}
	\begin{proof}
		Suppose that $A$ is diagonalizable.   Then $A = PDP^{-1}$, where $D = {\rm diag}(\lambda_1, \cdots, \lambda_n)$, $\lambda_1, \cdots, \lambda_n$ are dual complex numbers.   By definition, $D$ has exactly $n$ eigenvalues $\lambda_1, \cdots, \lambda_n$, with eigenvectors $\ve_1, \cdots, \ve_n$.   By Proposition \ref{p3.2}, $A$ has exactly $n$ eigenvalues $\lambda_1, \cdots, \lambda_n$, with eigenvectors $\vx_1=P\ve_1, \cdots, \vx_n=P\ve_n$.   Let $P = P_s + P_d\epsilon$ and $\vx_j = \vx_{js} + \vx_{jd}\epsilon$ for $j = 1, \cdots, n$.   Then $P_s =(\vx_{1s}\ \ \cdots\ \vx_{ns})$.   By Proposition \ref{p3.1}, $P_s^{-1}$ exists.
		This implies that  $\vx_{1s}, \cdots, \vx_{ns}$ are linearly independent, i.e., $\vx_1, \cdots, \vx_n$ are appreciably linearly independent.
		
		On the other hand, suppose that $A$ has $n$ eigenvalues $\lambda_1, \cdots, \lambda_n$, with eigenvectors $\vx_1, \cdots, \vx_n$, which are appreciably linearly independent.   Let $P = (\vx_1 \ \cdots \ \vx_n) = P_s +P_d\epsilon$.
		Then $P_s^{-1}$ exists.    By Proposition \ref{p3.1}, $P$ is invertible.  Then $A = PDP^{-1}$, where $D = {\rm diag}(\lambda_1, \cdots, \lambda_n)$, i.e., $A$ is diagonalizable.
	\end{proof}
	
	Note that if $A$ is diagonalizable, then $A_s$ must be diagonalizable, but the converse is not true.   See the following example.
	
	\medskip
	
	{\bf Example 3}  - a dual number matrix $A$, which is not diagonalizable, though $A_s$ is diagonalizable.

	Let $A = A_s+A_d\epsilon$, where
	\begin{equation}
		A_s = \left(\begin{array}{cc}
			1 & 0 \\
			0 & 1
		\end{array} \right)
		\text{ and }
		A_d = \left(\begin{array}{cc}
			1 & 1 \\
			0 & 1
		\end{array} \right).
	\end{equation}
	Then $A_s$ is the identity matrix, thus a diagonal matrix.  On the other hand, we may check that $A$ has only one eigenvalue $\lambda = 1 + \epsilon$, with eigenvectors {$\vx=[\alpha,0]^\top + \vx_d\epsilon$, where $\alpha$ is a nonzero complex number, and $\vx_d$ is any complex vector in ${\mathbb C}^2$.}  By Theorem \ref{t4.1}, $A$ is not diagonalizable.

	\bigskip
	
	\section{The Jordan Form of A Dual Complex Matrix with A Diagonalizable Standard Part}
	
	To know the situation in the general case, one has to explore the Jordan form theory of dual complex matrices.    Here, we need to point out the difficulty in establishing the Jordan form theory of dual complex matrices.   In the case of complex matrices, a key tool for establishing the Jordan theory is {Schur's} unitary triangularization theorem \cite{HJ12}.   In the case of the dual complex matrices, {triangularization} may not be possible, as there are examples of dual complex matrices which have no eigenvalues at all, but an upper triangular matrix has eigenvalues as its diagonal entries.

	However, we may start to explore the Jordan form theory of dual complex matrices in some special cases.
	We first consider the case that $A_s = \mathrm{diag}(\lambda_{1s}I_{n_1},\cdots,\lambda_{ts}I_{n_t})$,  $n_1+\cdots + n_t = n$  and $\lambda_{1s},\cdots,\lambda_{ts} \in {\mathbb C}$   are  distinct  complex numbers.
	For an $n\times n$ matrix $G$, denote the $(i,j)$-th subblock  as $G_{ij}$, which is corresponding to the  $\sum_{l=1}^{i-1}n_l+1,\dots, \sum_{l=1}^{i}n_l$ rows and the $\sum_{l=1}^{j-1}n_l+1,\dots, \sum_{l=1}^{j}n_l$ columns of $G$.
	
	Denote an $m \times m$ Jordan block by
	\begin{equation}
		J_m(\lambda) :=\left(\begin{array}{ccccc}
			\lambda  & 1 & 0 & \dots &  0 \\
			0 & \lambda &  1 & \dots &  0 \\
			\dots & \dots & \dots & \dots & \dots \\
			0  & 0 & 0 & \dots &  1 \\
			0  & 0 & 0 & \dots &  \lambda \\
		\end{array} \right).
	\end{equation}

	\begin{Lem}\label{lem5.1}
		Suppose that $A = A_s + A_d\epsilon \in {\mathbb {DC}}^{n \times n}$, where $A_s = \mathrm{diag}(\lambda_{1s}I_{n_1},\dots,\lambda_{ts}I_{n_t})$,  $n_1+\cdots + n_t = n$  and $\lambda_{1s},\dots,\lambda_{ts} \in {\mathbb C}$   are  distinct  complex numbers. Then let $P_{is}\in\mathbb C^{n_i\times n_i}$, $i=1,\dots,t$ be  invertible complex matrices such that {$J_{id} = P_{is}^{-1} A_{iid}P_{is}$} are Jordan  matrices of $A_{iid}$, and
		\begin{equation}
			P=\left(\begin{array}{cccc}
				P_{1s}  & O & \dots &  O \\
				O & P_{2s} &  \dots &  O \\
				\dots & \dots & \ddots& \dots \\
				O  &O & \dots &  P_{ts} \\
			\end{array} \right)
			+\left(\begin{array}{cccc}
				O  & P_{12d}\epsilon & \dots &  P_{1td}\epsilon \\
				P_{21d}\epsilon & O &  \dots &  P_{2td}\epsilon \\
				\dots & \dots &\ddots & \dots \\
				P_{t1d}{\epsilon}  & P_{t2d}\epsilon & \dots &  O \\
			\end{array} \right),
		\end{equation}
		where $P_{ijd} = (\lambda_{js}-\lambda_{is})^{-1}A_{ijd} P_{js}$ for $i\neq j$ and $P_{iid}=O_{n_i\times n_i}$.
		Then we have
		\begin{equation} \label{en13}
			P^{-1}AP = \mathrm{diag}\left(\lambda_{1s}I_{n_1}+J_{1d}\epsilon,\dots,\lambda_{ts}I_{n_t}+J_{td}\epsilon \right).
		\end{equation}
	\end{Lem}
	\begin{proof}
		Let $J_0=P^{-1}AP$.
		By direct computations, we have
		$$J_0=P_s^{-1}A_sP_s+ (P_s^{-1}A_sP_d+P_s^{-1}A_dP_s-P_s^{-1}P_dP_s^{-1}A_sP_s)\epsilon.$$
		Then it holds that
		\begin{eqnarray*}
			J_d & :=	 &P_s^{-1}A_sP_d+P_s^{-1}A_dP_s-P_s^{-1}P_dP_s^{-1}A_sP_s \\
			&=&P_s^{-1}A_sP_sP_s^{-1}P_d+P_s^{-1}A_dP_s-P_s^{-1}P_dP_s^{-1}A_sP_s\\
			&=& A_sP_s^{-1}P_d  + P_s^{-1}A_dP_s-P_s^{-1}P_dA_s, \ \
		\end{eqnarray*}
		where $P_s^{-1}A_sP_s=A_s$ in the last equality.
		Hence, 	 by  the choice of  $P_{is}$,    the diagonal subblocks $J_{id} = P_{is}^{-1}A_{iid}P_{is}$ are Jordan matrices.
		For $i\neq j$, by  the choice of  $P_{ijd}$, we have
		\begin{eqnarray*}
			J_{ijd}  &=&  \lambda_{is}P_{is}^{-1}P_{ijd}+P_{is}^{-1}A_{ijd}P_{js} -\lambda_{js} P_{is}^{-1}P_{ijd}\\
			&=& O_{n_i\times n_j}.
		\end{eqnarray*}
		Then (\ref{en13}) follows. 		This completes the proof.
	\end{proof}
	
	In the above lemma, assume that $J_{id} = \mathrm{diag}\left(J_{i1d},\dots,J_{ik_id}\right)$, where $J_{ijd}$ for $j = 1, \cdots, k_i$, are the Jordan blocks of $J_{id}$, in the form of $J_{ijd} = J_{m_{ij}}(\lambda_{ijd})$,  where $\lambda_{ijd}$ for $j = 1, \cdots, k_i, i = 1, \cdots, t$, are complex numbers.
	Then we have
	\begin{equation} \label{en16}
		J_0  = \mathrm{diag}\left(J_{11},\dots,J_{{1}k_1},J_{21}, \cdots, J_{tk_t}\right),
	\end{equation}
	where
	\begin{equation} \label{en17}
		J_{ij} =\left(\begin{array}{ccccc}
			\lambda_{is} + \lambda_{ijd}\epsilon  & \epsilon & 0 & \dots &  0 \\
			0 & \lambda_{is} + \lambda_{ijd}\epsilon &  \epsilon & \dots &  0 \\
			\dots & \dots & \dots &  \dots  & \dots \\
			0  & 0 & 0 & \dots &  \epsilon \\
			0  & 0 & 0 & \dots &  \lambda_{is} + \lambda_{ijd}\epsilon \\
		\end{array} \right),
	\end{equation}
	for $j = 1, \cdots, k_i, i = 1, \cdots, t$.   We call $J_0$ the {\bf Jordan matrix} of the dual complex matrix $A$, and $J_{ij}$, for $j = 1, \cdots, k_i, i = 1, \cdots, t$, the {\bf Jordan blocks} of $A$.
	
	In this way, we have the following theorem.
	\begin{Thm} \label{t5.2}
		Suppose that $B = B_s + B_d\epsilon \in {\mathbb {DC}}^{n \times n}$, where $B_s$ is diagonalizable,
		$B_s = Q_sA_sQ_s^{-1}$, $A_s = \mathrm{diag}(\lambda_{1s}I_{n_1},\dots,\lambda_{ts}I_{n_t})$ and $\lambda_{1s},\dots,\lambda_{ts} \in {\mathbb C}$   are distinct complex numbers.   Let $A_d = Q_s^{-1}B_dQ_s$.  Then $B$ has the Jordan form $J_0$ as expressed by (\ref{en16}) and (\ref{en17}), with {$J_0= P^{-1}Q_s^{-1}BQ_sP$.}
		Furthermore, $B$ has $k_1+\cdots + k_t$ eigenvalues
		$\lambda_{is} + \lambda_{ijd}\epsilon$, for $j = 1, \cdots, k_i$, $i=1,\dots,t$, with appreciably linearly independent eigenvectors ${\vx_{ijs}} + \vx_{ijd}\epsilon$, for $j = 1, \cdots, k_i$, $i=1,\dots,t$.
		Here,  {one choice of} $\vx_{ijs}+\vx_{ijd}\epsilon$ is the $\sum_{l_i=1}^{i-1}n_{l_i}+\sum_{l_j=1}^{j-1}m_{i,l_j}+1$ column of $Q_sP$, $n_0=0$, $m_{i,0}=0$ for $j = 1, \cdots, k_i$, $i=1,\dots,t$.
	\end{Thm}
	\begin{proof}  Let {$A = Q_s^{-1}BQ_s$.}   Then this theorem follows from Lemma \ref{lem5.1}, {Proposition \ref{p3.2},} and the argument before this theorem.
	\end{proof}

	\begin{Cor}\label{cor5.3}
		Suppose that $B = B_s + B_d\epsilon \in {\mathbb {DC}}^{n \times n}$. Then $B$ is diagonalizable if and only if
		the following two conditions hold:
		\begin{itemize}
			\item[(i)]  $B_s$ is diagonalizable. Namely, there exists an invertiable matrix $Q_s\in \mathbb C^{n\times n}$ and a diagonal matrix $A_s$ such that  $B_s = Q_sA_sQ_s^{-1}$, where  $A_s = \mathrm{diag}(\lambda_{1s}I_{n_1}$, $\dots,$ $\lambda_{ts}I_{n_t})$ and $\lambda_{1s},\dots,\lambda_{ts} \in {\mathbb C}$   are distinct complex numbers.
			\item[(ii)]  For $i=1,\dots,t$, there exist  invertiable matrices $P_{is}\in\mathbb C^{n_i\times n_i}$ such that
			\begin{equation*}
				J_{id} = P_{is}^{-1} A_{iid} P_{is}
			\end{equation*}
			are diagonal matrices,  where $A_d = Q_s^{-1}B_dQ_s$, $A_{iid}$ is the  subblock of {$A_d$} consisting of the  $\sum_{l=1}^{i-1}n_l+1,\dots, \sum_{l=1}^{i}n_l$ rows and columns of $A_d$,
			{where $n_0=0$.}
		\end{itemize}
		
	\end{Cor}

	{Suppose that $A_s$ is diagonalizable. Then {the condition that} $A_d$ is diagonalizable {is neither sufficient nor necessary} for  $A$ to be diagonalizable.  See the following examples.
		
		\medskip
		
		{\bf Example 4}  - a dual number matrix $A$, which is not diagonalizable, though both $A_s$ and $A_d$ are diagonalizable.
		
		Let $A = A_s+A_d\epsilon$, where
		\begin{equation}
			A_s = \left(\begin{array}{ccc}
				1 & 0  & 0\\
				0 & 1 & 0 \\
				0 & 0 & 2
			\end{array} \right)
			\text{ and }
			A_d = \left(\begin{array}{ccc}
				1 & 1 & 1 \\
				0 & 1 & 0 \\
				1 & 0 & 1
			\end{array} \right).
		\end{equation}
		Then $A_s$ is a diagonal matrix. $A_d$ has eigenvalues $0, 1, 2$, thus diagonalizable. On the other hand, we have
		$$P_s=I_3, \ P_d=\left(\begin{array}{ccc}
			0 & 0 & 1 \\
			0 & 0 & 0 \\
			-1 & 0 & 0
		\end{array} \right), \
		J_s=A_s, \text{ and } J_d = \left(\begin{array}{ccc}
			1 & 1 & 0 \\
			0 & 1 & 0 \\
			0 & 0 & 1
		\end{array} \right).$$
		Then we may check that $A$ has only two eigenvalues $\lambda_1 = 1 + \epsilon$  and  $\lambda_2 = 2 + \epsilon$.
		By Theorem \ref{t4.1}, $A$ is not diagonalizable.
		
		\medskip
		
		{\bf Example 5}  - a diagonalizable dual number matrix $A$,   yet   $A_d$ is not diagonalizable.
		
		Let $A = A_s+A_d\epsilon$, where
		\begin{equation}
			A_s = \left(\begin{array}{ccc}
				1 & 0  & 0\\
				0 & 1 & 0 \\
				0 & 0 & 2
			\end{array} \right)
			\text{ and }
			A_d = \left(\begin{array}{ccc}
				1 & 0 & 0 \\
				0 & 1 & 1 \\
				1 & 0 & 1
			\end{array} \right).
		\end{equation}
		Then $A_s$ is a diagonal matrix. $A_d$ has only one eigenvalue  $1$, thus is not diagonalizable. On the other hand,   we have
		$$P_s=I_3, \ P_d=\left(\begin{array}{ccc}
			0 & 0 & 0 \\
			0 & 0 & 1 \\
			-1 & 0 & 0
		\end{array} \right), \
		J_s=A_s, \text{ and } J_d = I_3.$$
		Thus $A$ is diagonalizable since $J=P^{-1}AP$ is diagonal.
	}

	\bigskip
	
	{\section{The Jordan Form of A Dual Complex Matrix with A Jordan Block Standard Part}
		
		We now consider the Jordan form of a dual complex matrix $A = A_s + A_d\epsilon \in {\mathbb {DC}}^{n \times n}$ for $n \ge 2$, whose standard part is a Jordan block in the form of
		$A_s  =J_n(\lambda_s)$,
		where $\lambda_s$ is a fixed complex number.    We now consider $J = P^{-1}AP$ with $P = P_s + P_d\epsilon \in {\mathbb {DC}}^{n \times n}$, to see what kind of $P$ makes $J$ as simple as possible.    Write $J = J_s + J_d\epsilon \in {\mathbb {DC}}^{n \times n}$.   Then $J_s = P_s^{-1}A_sP_s$.  As $A_s$ is already a standard simple form, we have to take $P_s = I_n$ and $J_s = A_s$.  Then the question is:  what kind of $P_d$ makes $J_d$ as simple as possible?
		We have the following theorem.
		
		\begin{Thm}\label{t6.1}
			Suppose that $A = A_s + A_d\epsilon \in {\mathbb {DC}}^{n\times n}$ is a dual complex matrix for $n \ge 2$, with $A_s = J_n(\lambda_s)$, and $\lambda_s$ a fixed complex number there.
			Let $P_d\in\mathbb C^{n\times n}$ with
			\begin{equation}
				p_{j+l+1,j,d}=\left\{
				\begin{array}{ll}
					-a_{j+l,j,d},  &  \text{ if } j=1,0\le l\le n-2;\\
					p_{j+l,j-1,d}- a_{j+l,j,d},   &   \text{ if } 2\le j\le n-l-1, 0\le l\le n-2;\\
					p_{j+l,j-1,d}-a_{i+l,j,d},   &  \text{ if } 1-l\le j\le n, 1-n\le l\le -1;\\
					0,  &  \text{ otherwise.}\\
				\end{array}
				\right.
			\end{equation}
			Then for  $J = P^{-1}AP$ with $P = I_n + P_d\epsilon \in {\mathbb {DC}}^{n \times n}$, we have $J_s=A_s$ and
			\begin{equation} \label{en25}
				J_d =\left(\begin{array}{cccc}
					0 & 0 & \dots & 0 \\
					\vdots &  \vdots &  \vdots &  \vdots \\
					0 & 0 & \dots & 0 \\
					f_{n1d} & f_{n2d} & \dots & f_{nnd}
				\end{array} \right),
			\end{equation}
			where $f_{n1d} = a_{n1d}$, $f_{nid} = a_{nid}-p_{n,i-1,d}$ for $i = 2, \cdots, n$.
		\end{Thm}
		\begin{proof}
			By the definition, we have
			\begin{equation} \label{en19}
				J_d = A_d + A_sP_d - P_dA_s.
			\end{equation}
			We may choose $P_d$ such that the entries of $J_d$ have as many zeros as possible.  Let $A_d = \left(a_{ijd}\right)$, $J_d = \left(f_{ijd}\right)$ (we use letter $f$ instead of $j$ for the entries of $J$ to avoid confusing with the index $j$) and $P_d=\left(p_{ijd}\right)$.
			The matrix equation \eqref{en19} contains $n^2$ linear equations.   According to the diagonal and sub-diagonals of that matrix equation, we may partition these $n^2$ linear equations into $2n-1$ groups.
			
			For an $n \times n$ matrix $G = (g_{ij})$, denote  the $l$-th (sub)diagonal vector as an $n_l$-dimensional vector $$G^{(l)}=[g_{l+1,1},g_{l+2,2},\dots,g_{n,n-l}], \text{ for } l=n-1,\dots,0,$$
			and
			$$G^{(l)}=[g_{1,1-l},g_{2,2-l},\dots,g_{n+l,n}], \text{ for } l=-1,\dots,1-n,$$
			respectively. Here, $n_l=n-|l|$.
			
			Then  the $n^2$ linear equations contained in  \eqref{en19} can be  partitioned  into $2n-1$ groups of linear equations   of $P_d^{(l)}$ for $l=n-1,\dots,1-n$  respectively.
			Now, consider the following three cases of these $2n-1$ groups.
			
			Case (i).  $l=n-1$. {We} have
			$$J_d^{(l)} = A_d^{(l)},$$
			or equivalently, {$f_{n1d}=a_{n1d}$}. Hence, if $a_{n1d}\neq0$, $f_{n1d}$ cannot be equal to zero no matter what $P_d$ we choose.
			
			Case (ii). $l=n-2,\dots,0$. {We} have
			\begin{equation}\label{equ:case2}
				J_d^{(l)} = A_d^{(l)} + B_l P_d^{(l+1)},
			\end{equation}
			where $$B_l=\left(\begin{array}{c}
				I_{n_{l}-1}   \\
				O_{1\times (n_{l}-1)}
			\end{array}\right)- \left(\begin{array}{c}
				O_{1\times (n_{l}-1)} \\
				I_{n_{l}-1}
			\end{array}\right)\in\mathbb R^{n_l\times (n_l-1)}.$$
			$B_l$ is of full column rank.
			Hence, \eqref{equ:case2} is {an} overdetermined system of $P_d^{(l+1)}$.
			The linear system
			$A_d^{(l)} + B_l P_d^{(l+1)}= O_{n_l\times 1}$
			has a solution if and only if
			\begin{equation}\label{eqn:A_dinspan}
				A_d^{(l)} \in \span\left(B_l \right).
			\end{equation}
			
			Hence,  $J_d^{(l)}=O_{n_l\times 1}$ may not hold.
			However, the linear system $\left(J_d^{(l)}\right)_{1:n_l-1}=O_{(n_l-1)\times 1}$  has  a unique solution as
			$\left(P_{d}^{(l+1)}\right)_1=- \left(A_{d}^{(l)}\right)_1$, and
			$$\left(P_{d}^{(l+1)}\right)_i= \left(P_{d}^{(l+1)}\right)_{i-1}-\left(A_{d}^{(l)}\right)_i, i=2,\dots,n_l-1.$$
			Equivalently,
			\begin{equation}
				p_{l+2,1d} = - a_{l+1,1d}, \ p_{l+i+1,i,d} =  p_{l+i,i-1,d}- a_{l+i,id}, i=2,\dots,n_l-1.
			\end{equation}
			Then we have  $f_{l+i,id}=0$ for all $i=1,\dots,n-l-1$, and
			$f_{n,n-l,d} = a_{n,n-l,d}-p_{n,n-l-1,d}$  for all $l=n-2,\dots,0$.
			
			{Case} (iii).  $l=-1,\dots,1-n$. Then  we have
			\begin{equation}\label{equ:case3}
				J_d^{(l)} = A_d^{(l)} +B_lP_d^{(l+1)},
			\end{equation}
			where
			{$$B_l=\left(\begin{array}{cc}
					O_{n_{l}\times 1} & I_{n_{l}}
				\end{array}\right) - \left(\begin{array}{cc}
					I_{n_{l}}  & O_{n_{l}\times 1}
				\end{array}\right)\in\mathbb R^{n_l\times(n_l+1)}.$$}
			$B_l$ is of full row rank.
			In this case, \eqref{equ:case3} is an underdetermined system. Hence,  we may  have infinitely many  choices to make $J_d^{(l)}= O_{n_l\times 1}$.
			
			Specifically, we set $\left(P_{d}^{(l+1)}\right)_1=0$,
			$$\left(P_{d}^{(l+1)}\right)_i=\left(P_{d}^{(l+1)}\right)_{i-1}- \left(A_{d}^{(l)}\right)_{i-1}, i=2,\dots,n_l+1.$$
			Then we have  $J_{d}^{(l)}=O_{n_l\times 1}$ for all  $l=-1,\dots,1-n$.
			Equivalently,  we set $p_{1,-l,d}=0$,
			\begin{equation}
				p_{i,i-l-1d} =   p_{i-1,i-l-2d}- a_{i-1,i-l-1d}, i=2,\dots,n+l+1.
			\end{equation}
			Then we have  $f_{i,i-l,d}=0$  for all $i=1,\dots,n+l$ and $l=-1,\dots,1-n$.
			
			This completes the proof.
		\end{proof}

		We may further write  $J =\lambda_s I + J_r$, where
		\begin{equation} \label{en26}
			J_r=\left(\begin{array}{ccccc}
				0 & 1 & 0 & \dots & 0 \\
				0 & 0 & 1 &\dots & 0\\
				\vdots &  \vdots &\vdots &  \vdots &  \vdots \\
				0 & 0 & 0&\dots & 1 \\
				f_{n1d}\epsilon & f_{n2d}\epsilon & f_{n3d}\epsilon & \dots & f_{nnd}\epsilon
			\end{array} \right)
		\end{equation}
		{is a  companion matrix.  We call $J$ the {\bf Jordan matrix}}}
	of the dual complex matrix $A = A_s + A_d\epsilon$, where $A_s= J_n(\lambda_s)$.
	
	Now, by using Theorem \ref{t3.4}, we have the following proposition.
	
	\begin{Prop}\label{prop6.2}
		Suppose that $A = A_s + A_d\epsilon \in {\mathbb {DC}}^{n\times n}$ is a dual complex matrix for $n \ge 2$, with $A_s= J_n(\lambda_s)$, and $\lambda_s$ a fixed complex number there.    If $a_{n1d} \not = 0$, then $A$ has no eigenvalue at all.   Otherwise, $A$ has eigenvalues $\lambda_s + \lambda_d\epsilon$, where $\lambda_d$ is any complex number.
	\end{Prop}
	
	Examples 1 and 2 confirm the conclusion of this proposition.
	
	\bigskip
	
	\section{Eigenvalues of A Square Dual Complex Matrix}
	
	We first consider the eigenvalues of a dual complex matrix $B = B_s + B_d\epsilon \in {\mathbb {DC}}^{n \times n}$,  whose standard part $B_s$ has only one eigenvalue $\lambda_s$. Then we may assume that $B=Q_sAQ_s^{-1}\in {\mathbb {DC}}^{n \times n}$,  and $A = A_s + A_d\epsilon \in {\mathbb {DC}}^{n \times n}$, where $A_s=\mathrm{diag}(\lambda_s I_{n_0},J_{n_1}(\lambda_{s}),\dots,J_{n_t}(\lambda_{s}))$.
	Here, $Q_s$ is an invertible $n \times n$ complex matrix.   

	Then we have the following theorem.
	
	\begin{Thm}\label{t7.1}
		Let $B=Q_sAQ_s^{-1}\in {\mathbb {DC}}^{n \times n}$,  $A = A_s + A_d\epsilon \in {\mathbb {DC}}^{n \times n}$, where $A_s=\mathrm{diag}(\lambda_s I_{n_0},J_{n_1}(\lambda_{s}),\dots,J_{n_t}(\lambda_{s}))$,   $\lambda_{s}$  is a fixed complex number,  $n_0+n_1 + \dots+n_t = n$, $n_i \ge 2$ for $i=1,\dots,t$.   Let   $m_i=n_0+n_1+\dots+n_{i-1}+1$   for $i=1,\dots,{t}$.
		{Denote $Z=[\ve_1,\dots,\ve_{n_0},\ve_{m_1}, \cdots, \ve_{m_t}]\in\mathbb R^{n\times (n_0+t)}$, $W=[\ve_1,\dots,\ve_{n_0},\ve_{m_2-1}, \cdots, \ve_{m_t-1},{\ve_n}]\in\mathbb R^{n\times (n_0+t)}$, and $C=W^\top A_d Z\in\mathbb C^{(n_0 + t) \times (n_0 +t)}$.}
		Denote $C$ as a block matrix:
		$$C = \left(\begin{array}{cc}
			C_{00} & C_{01} \\
			C_{10} & C_{11}
		\end{array} \right),$$
		where $C_{00}$ is {an} $n_0 \times n_0$ complex matrix, and $C_{11}$ is a $t \times t$ complex matrix.   Then $C_{00} = A_{00d}$.
		If $B$ has an eigenvalue $\lambda$,   we have $\lambda = \lambda_s + \lambda_d \epsilon$ for a complex number $\lambda_d$. And $B$ has an eigenvalue $\lambda = \lambda_s + \lambda_d \epsilon$, if and only if the following system
		\begin{equation} \label{en28}
			\left(\begin{array}{cc}
				\lambda_dI_{n_0}-C_{00} & -C_{01} \\
				-C_{10} & -C_{11}
			\end{array} \right)\vy = \0,
		\end{equation}
		has a nonzero solution $\vy$, and a complex number solution $\lambda_d$.
		Furthermore, we have $\vx_s=Z\vy$,  and one choice of $\vx_d$ is that {$\vx_d=(A_s - \lambda_s I_n)^\top(\lambda_dI_n -A_d)\vx_s$.}
	\end{Thm}
	\begin{proof}
		
		The eigenvalues of $B$ are the same as the eigenvalues of $A$, as they are similar. Suppose that $A$ has an eigenvalue $\lambda = \lambda_s + \lambda_d\epsilon$ with an eigenvector $\vx = \vx_s + \vx_d\epsilon$.  By (\ref{en2}), $\lambda_s$ is an eigenvalue of $A_s$ with an eigenvector $\vx_s$.  By the formulation of  $A_s$, $\vx_s$ can be any nonzero vector in the space spanned by {$Z$.}
		Thus, we may extract $\vy \in {\mathbb C}^n$ from $\vx_s$ such that $\vx_s = {Z\vy}$.
		
		Now, consider (\ref{en3}).   We have
		\begin{equation} \label{en29}
			(A_s - \lambda_s I_n)\vx_d = (\lambda_dI_n -A_d)\vx_s.
		\end{equation}
		By the formation of $A_s$, we have $$(A_s - \lambda_s I_n)\vx_d =
		\left(0, \cdots, 0, x_{m_1+1, d}, \cdots, x_{m_2-1, d}, 0, x_{m_2+1, d}, \cdots, x_{nd}, 0\right)^\top.$$    Comparing this with (\ref{en29}), we find that equations $m_1, m_1+1, \cdots, m_2-2, m_2$, $\cdots$, $n-1$ of (\ref{en29}) are always satisfied, as we may let $x_{m_1+1, d}, \cdots, x_{m_2-1, d}$, $x_{m_2+1, d}$, $\cdots$, $x_{nd}$ take values such that these equations are satisfied.  Thus, (\ref{en29}) has a nonzero solution $\vx_s$, which is an eigenvector of $A_s$, if and only if (\ref{en28}) has a nonzero solution $\vy$.  The conclusion follows now.
	\end{proof}
	
	If $t=0$, then (\ref{en28}) indicates that $\lambda_d$ should be an eigenvalue of $C = C_{00} = A_{00d} = A_d$, and $\vy$ is a corresponding eigenvector.
	
	We now consider the general case.

	In general, denote $K_n(\lambda_s;n_0,\cdots,n_t)=\mathrm{diag}(\lambda_s I_{n_0},J_{n_1}(\lambda_{s}),\dots,J_{n_t}(\lambda_{s}))$,   where  $\lambda_{s}$ is a fixed complex number,  $n_0+n_1 + \dots+n_t = n$, $n_j \ge 2$ for $j=1,\dots,t$.
	
	\begin{Thm}\label{t7.2}
		Let $B=Q_sAQ_s^{-1}\in {\mathbb {DC}}^{n \times n}$,  $A = A_s + A_d\epsilon \in {\mathbb {DC}}^{n \times n}$, where
		$$A_s=\mathrm{diag}(K_{n_1}(\lambda_{1s};n_{10},\dots,n_{1t_1}), \cdots,K_{n_p}(\lambda_{ps};n_{p0},\dots,n_{pt_p})),$$
		$\lambda_{1s},\dots,\lambda_{ps}$ are distinct  complex numbers,  $n_1 + \dots+n_p = n$, and $n_0=0$.
		Denote		$E_i$ as the $(n_1+\cdots+n_{i-1}+1,\cdots,n_1+\cdots+n_i)th$ columns of $I_n$,
		$\vx_{is} = E_i^\top\vx_s$, $\vx_{id} = E_i^\top\vx_d$,
		$A_{is}=E_i^\top A_s E_i$, $A_{id}=E_i^\top A_d E_i$, {$A_{kid}=E_k^\top A_d E_i$}.
		Let   $m_{i,j}=n_{i0}+n_{i1}+\dots+n_{i,j-1}+1$   for $j=1,\dots,t_i$, $Z_i=[\ve_1,\dots,\ve_{n_{i0}},\ve_{m_{i1}},$ $\cdots,$ $\ve_{m_{it_i}}]$ $\in\mathbb R^{n_i\times (n_{i0}+t_i)},$ $W_i=[\ve_1,\dots,\ve_{n_0},\ve_{m_{i2}-1}, \cdots, \ve_{m_{it_i}-1},{\ve_{n_i}}]\in\mathbb R^{n_i\times (n_{i0}+t_i)}$, and $C_i$ $=$ $W_i^\top A_{id} Z_i$$\in\mathbb C^{(n_{i0} + t_i) \times (n_{i0} +t_i)}$, {for $i=1,\cdots,p$.}
		Denote $C_i$ as a block matrix:
		$$C_i = \left(\begin{array}{cc}
			C_{i00} & C_{i01} \\
			C_{i10} & C_{i11}
		\end{array} \right),$$
		where $C_{i00}$ is an $n_{i0} \times n_{i0}$ complex matrix, and $C_{i11}$ is a $t_i \times t_i$ complex matrix.
		We have the following conclusions.
		\begin{itemize}
			\item[(i)]  If $B$ has an eigenvalue $\lambda$,    then
			there exists $i\in\{1,\cdots,p\}$ such that $\lambda = \lambda_{is} + \lambda_{d} \epsilon$ for a complex number $\lambda_d$.
			\item[(ii)] Given an index $i$ in $1,\cdots,p$, $B$ has an eigenvalue $\lambda = \lambda_{is} + \lambda_d \epsilon$, if and only if the following system
			\begin{equation} \label{en30}
				\left(\begin{array}{cc}
					\lambda_dI_{n_{i0}}-C_{i00} & -C_{i01} \\
					-C_{i10} & -C_{i11}
				\end{array} \right)\vy = \0,
			\end{equation}
			has a nonzero solution $\vy$, and a complex number solution $\lambda_d$.
			Furthermore, we have $\vx_s=E_i\vx_{is}$, $\vx_{is} = Z_i\vy$,    one choice of $\vx_{id}$ is $\vx_{id}= (A_{is} - \lambda_{is} I_{n_i})^\top(\lambda_dI_{n_i} -A_{id})\vx_{is}$,
			and
			$$\vx_d =  E_i\vx_{id} {- \sum_{k=1, k\neq i}^pE_k(A_{ks} - \lambda_{is} I_{n_k})^{-1}A_{kid}\vx_{is}.}$$
		\end{itemize}
		
	\end{Thm}
	
	{\begin{proof}
			Suppose $\lambda_s=\lambda_{is}$ for a given index $i$ in $1,\cdots,{p}$.
			{By the block diagonal structure of $A_s$, we have $\vx_s=E_i\vx_{is}$.}
			Equation \eqref{en29} can be divided into $p$  systems of equations:
			\begin{equation} \label{en31}
				E_k^\top(A_s - \lambda_{is} I_n)\vx_d = E_k^\top(\lambda_dI_n -A_d)E_i\vx_{is},
			\end{equation}
			for $k=1,\cdots,p$.  By rewriting $\vx_d$ as $\vx_d=\sum_{l=1}^nE_l\vx_{ld}$, the left hand side of \eqref{en31} satisfies
			$$E_k^\top(A_s - \lambda_{is} I_n)\vx_d=\sum_{l=1}^nE_k^\top(A_s - \lambda_{is} I_n)E_l\vx_{ld}=(A_{ks} - \lambda_{is} I_{n_k})\vx_{kd},$$
			where   the last equality follows from $E_k^\top(A_s - \lambda_s I_n)E_l=O_{n_k\times n_l}$ for all $k\neq l$.
			For the left hand side of \eqref{en31}, we have
			$$E_k^\top(\lambda_dI_n -A_d)E_i\vx_{is}=(\lambda_d \delta_{ki}I_{n_i} -A_{kid})\vx_{is}.$$
			Here, $\delta_{ki}=1$ if $k=i$ and  $\delta_{ki}=0$ otherwise. Equation \eqref{en31} can be reduced to
			\begin{equation}\label{en32}
				(A_{ks} - \lambda_{is} I_{n_k})\vx_{kd} = (\lambda_d \delta_{ki}I_{n_i} -A_{kid})\vx_{is}.
			\end{equation}
			
			If $k=i$, equation \eqref{en32} is equivalent to
			$(A_{is} - \lambda_{is} I_{n_i})\vx_{id} = (\lambda_dI_{n_i} -A_{id})\vx_{is}$. Hence, the formulations of $\vx_{is}$ and $\vx_{id}$ follow  directly from Theorem~\ref{t7.1}.
			
			If $k\in \{1,\cdots,p\} \setminus \{i\}$, equation \eqref{en32} is equivalent to
			$$ (A_{ks} - \lambda_{is} I_{n_k})\vx_{kd} = -A_{kid}\vx_{is}.$$
			Since $\lambda_{ks}\neq \lambda_{is}$, the coefficient matrix on the left hand side is nonsingular. Hence, we have
			$$\vx_{kd} = -(A_{ks} - \lambda_{is} I_{n_k})^{-1}A_{kid}\vx_{is}.$$
			
			This completes the proof.
		\end{proof}
		
	}

	\bigskip
	\section{Further Discussion}
	
	In this paper, we explored the eigenvalue theory of dual complex matrices.   Unlike complex matrices, a square dual complex matrix may have no eigenvalues at all, or have infinitely many eigenvalues.
	By exploring the Jordan form theory of some special dual complex matrices in Sections 5 and 6.  In Section 7, we give a description of the eigenvalues of a general square dual complex matrix.  Our work is only a starting point.  We may further explore on these.  Is there a Jordan form for a general square dual complex matrix? We may also explore the determinant and characteristic polynomial theory of dual complex matrices.
	Then we may establish a sound dual complex matrix theory, like {and different from} the cases of complex matrices \cite{HJ12} and quaternion matrices \cite{Ji19, WLZZ18, Zh97}.

	We also need to find more applications of dual complex matrices.  As in \cite{WDW23}, the dual part of a dual complex matrix can be viewed as a perturbation, we may also explore possible applications of dual complex matrices in ocean science and imaging science.
	
	As computational methods of computing eigenvalues of a dual quaternion Hermitian matrix have already been investigated in \cite{CQ23}, computational methods for computing eigenvalues of a dual complex diagonalizable matrix are surely feasible.

	\bigskip
	
	

	\bigskip
	


\end{document}